\documentclass[12pt]{amsart} 

\usepackage[margin=3cm]{geometry}
\usepackage[english]{babel}
\usepackage{hyperref}
\usepackage{amsmath}
\usepackage{amsthm}
\usepackage{amssymb}
\usepackage{bbm}
\usepackage{ucs}
\usepackage{enumerate}
\usepackage{tikz}
\usepackage{tikz-cd}
\usepackage[utf8]{inputenc}
\usepackage[T1]{fontenc}
\usepackage{xcolor}
\usepackage{mathtools}
\usepackage{comment}
\usepackage{dsfont}
\usepackage{faktor}
\usepackage{xcolor}
\definecolor{deepsea}{RGB}{0, 139, 139}

\hypersetup{
	colorlinks=true,       
	linkcolor=blue,    
	citecolor=cyan,    
	urlcolor=cyan      
}
\usepackage[draft,inline,nomargin,index]{fixme}

\fxsetup{theme=color,mode=multiuser}
\FXRegisterAuthor{pf}{apf}{\color{olive}PF}
\FXRegisterAuthor{flm}{aflm}{\color{red}FLM}
\FXRegisterAuthor{km}{akm}{\color{purple}KM}
\FXRegisterAuthor{ip}{aip}{\color{blue}IP}
\AtBeginDocument{%
	\def\MR#1{}
}

\numberwithin{equation}{section}

\newtheorem{introtheorem}{Theorem}

\newtheorem{theorem}{Theorem}[section]

\newtheorem{lemma}[theorem]{Lemma}
\newtheorem{proposition}[theorem]{Proposition}
\newtheorem{corollary}[theorem]{Corollary}

\theoremstyle{definition}
\newtheorem{definition}[theorem]{Definition}

\newtheorem{remark}[theorem]{Remark}

\newcommand{\R}{\mathbb{R}}

\newcommand{\C}{\mathbb{C}}
\newcommand{\N}{\mathbb{N}}

\newcommand{\id}{{\rm id}}

\DeclareMathOperator{\re}{Re}
\DeclareMathOperator{\im}{Im}
\DeclareMathOperator{\sa}{sa}
\DeclareMathOperator{\op}{op}
\DeclareMathOperator{\dom}{dom}

\DeclarePairedDelimiter{\ip}{\langle}{\rangle}

\newcommand{\norm}[1]{\left\| #1 \right\|}

\title{Simplicity of $q$-Gaussian $\mathrm{C}^*$-algebras}
\date{}
\author[Amrutam]{Tattwamasi Amrutam}
\address{Tattwamasi Amrutam
	\newline
	Institute of Mathematics of the Polish Academy of Sciences, ul. \'{S}niadeckich 8, 00-656, Warszawa, Poland}
\email{tattwamasiamrutam@gmail.com}
\author[Jekel]{David Jekel}
\address{David Jekel
	\newline
	Department of Mathematical Sciences, University of Copenhagen
	Universitetsparken 5, 2100 København (Denmark)
}
\address[Current address]{
	Department of Mathematics and Statistics, University of Ottawa, 150 Louis-Pasteur Pvt, Ottawa, ON, K1N 6N5, Canada
}
\email{djekel@uottawa.ca}
\author[Wasilewski]{Mateusz Wasilewski}
\address{Mateusz Wasilewski
	\newline 
	Institute of Mathematics of the Polish Academy of Sciences, ul. \'{S}niadeckich 8, 00-656, Warszawa, Poland}
\email{mwasilewski@impan.pl}
\begin{document}
	
	\begin{abstract}
		We show that $q$-Gaussian $\mathrm{C}^*$-algebras for $q \in (-1,1)$ have the Dixmier averaging property, and hence are simple with a unique trace.  We argue by combining rapid decay and spectral gap estimates for the commutators with the generators, which are obtained from free probability.
	\end{abstract}

	\maketitle
	
	\section{Introduction}
	
	The $q$-commutation relations $a_{i} a_{j}^{\ast} - q a_{j}^{\ast} a_{i} = \delta_{ij} \mathds{1}$ for $q\in [-1,1]$, first considered in \cite{FrischBourret1970}, are a natural generalization of the relations satisfied by bosonic ($q = 1$) and fermionic ($q = -1$) creation and annihilation operators.  Bo{\.z}ejko and Speicher \cite{BozejkoSpeicher1994} constructed operators on a Hilbert space satisfying these relations, and introduced the $q$-Gaussian algebras generated by the operators $a_{i} + a_{i}^{\ast}$. Since there is a natural trace on the $q$-Gaussian algebra, it can be completed in two different ways: to a $\mathrm{C}^*$-algebra and to a tracial von Neumann algebra. These $q$-Gaussian operator algebras were first seriously investigated in \cite{BozejkoKummererSpeicher1997}, where the authors showed simplicity of the $\mathrm{C}^{\ast}$-algebra and factoriality of the von Neumann algebra in the case of infinitely many generators.  However, most of the recent results on $q$-Gaussians have been about the von Neumann algebras, including important structural properties such as factoriality \cite{Ricard2005}, non-injectivity \cite{Nou2004}, complete metric approximation property and strong solidity \cite{Avsec2011}, and maximal amenability of the generator MASA for $|q| < 1/9$ \cite{PSW2018}.  Several of these results have also been generalized to the setting of $q$-Araki Woods factors, a natural non-tracial analog of $q$-Gaussian algebras \cite{KSW2023factoriality,Yang2024twisted,mandal2025simplicity}.
	
	Many of these results were first established for $q=0$, in which the $q$-Gaussian operators are a free semicircular family, and hence the von Neumann algebra is isomorphic to a free group factor. For this reason, the $q$-Gaussian algebras are often viewed as deformations of free group factors.  In fact, in the case of finitely many generators, the $q$-Gaussian algebras are isomorphic to free group von Neumann algebras for sufficiently small $|q|$ (depending on the number of generators) by \cite{GuionnetShlyakhtenko2014}; this was also generalized to the $q$-Araki--Woods and twisted Araki--Woods settings by Nelson \cite{Nelson2015transport} and Yang \cite{Yang2024twisted}.  On the other hand, for infinitely many generators the $q$-Gaussian von Neumann algebra is not isomorphic to a free group factor as soon as $q \neq 0$ \cite{Caspers2023}; the non-isomorphism of the $\mathrm{C}^*$-algebras was obtained earlier in \cite{BorstCaspersKlisseWasilewski2023}.  The isomorphism problem for finitely many generators and large $q$ remains open.
	
	Overall, compared with the litany of results for $q$-Gaussian von Neumann algebras, surprisingly little attention has been given to the $q$-Gaussian $\mathrm{C}^*$-algebras, with only a few results known in the case of infinitely many generators. This paper will establish several basic results about $q$-Gaussian $\mathrm{C}^*$-algebras for finitely many generators, namely, simplicity, uniqueness of the trace, and $\mathrm{C}^*$-algebraic Dixmier averaging.  We use free probability techniques to establish the von Neumann algebraic spectral gap property, and then upgrade this to Dixmier averaging for $\mathrm{C}^*$-algebras using the rapid decay property. To present the main results, we need to recall some useful definitions.  We assume familiarity with the basic theory of $\mathrm{C}^*$-algebras and tracial von Neumann algebras.
	
	\begin{definition}
		Let $A$ be a unital $\mathrm{C}^*$-algebra with a faithful trace $\tau$, and let $\norm{\cdot}_{\tau}$ denote the GNS norm.  Let $x = (x_j)_{j=1}^d$ in $A$.  We say that $(A,\tau,x)$ has \emph{spectral gap} if there exists a constant $C$ such that for all $y \in A$,
		\[
		\norm{y - \tau(y) 1}_{\tau} \leq C \sum_{j=1}^d (\norm{[y,x_j]}_\tau + \norm{[y^*,x_j]}_\tau).
		\]
	\end{definition}
	
	While the rapid decay property has mostly been studied in the group context (as well as for quantum groups), the definition naturally generalizes to tracial $*$-algebras (see \cite{hayes2025selfless}).
	
	\begin{definition}[{See \cite[Proposition 3.4]{Miyagawa2023} and \cite[Definition~3.2]{hayes2025selfless}}]
		Let $A$ be a unital $\mathrm{C}^*$-algebra with a faithful trace $\tau$ and let $x = (x_j)_{j=1}^d$ generate $A$.  Let
		\[
		\rho(n) = \sup 
		\left\{ \frac{\norm{f(x)}}{\norm{f(x)}_{\tau}}: f \text{ $*$-polynomial of degree } \leq n, f(x) \neq 0 \right\}.
		\]
		Then $(A,\tau,x)$ is said to have \emph{rapid decay} if $\rho(n)$ is bounded above by a polynomial in $n$.  It is said to have \emph{semirapid decay} if $\rho(n)$ grows subexponentially, that is, $\lim_{n \to \infty} \rho(n)^{1/n} = 1$.
	\end{definition}
	
	While rapid decay provides the necessary control over polynomial bounds, our primary goal relies on the Dixmier averaging property, which provides a direct route to proving simplicity and the uniqueness of the tracial state.
	\begin{definition}[{See e.g.\ \cite{archbold2017dixmier}}]
		A unital $\mathrm{C}^*$-algebra $A$ is said to have the \emph{Dixmier averaging property} if for every $y \in A$ and $\varepsilon > 0$, there exists $n \in \N$, unitaries $u_1$, \dots, $u_n \in A$, and $\lambda \in \C$ such that
		\[
		\norm{\frac{1}{n} \sum_{j=1}^n u_j y u_j^* - \lambda} < \varepsilon.
		\]
	\end{definition}
	
	We recall that if $A$ has the Dixmier averaging property, then $A$ can have at most one tracial state, since any tracial state $\tau$ must satisfy $\frac{1}{n} \sum_{j=1}^n \tau(u_j y u_j^*) = \tau(y)$ but the constant $\lambda$ only depends on $\varepsilon$ and not on $\tau$.  Moreover, if $A$ has a faithful trace $\tau$, then the Dixmier averaging property implies that $A$ is simple, because if $I$ is a closed ideal containing a nonzero element $x$, then by Dixmier averaging $\tau(x^*x) \in I$.  We will obtain Dixmier averaging for the $q$-Gaussian $\mathrm{C}^*$-algebras using the following general result that spectral gap and rapid decay imply Dixmier averaging.
	
	\begin{introtheorem} \label{thm: general Dixmier averaging}
		Let $A$ be a $\mathrm{C}^*$-algebra with a faithful trace $\tau$.  Suppose that $x = (x_j)_{j=1}^d$ generates $A$ and that $(A,\tau,x)$ has a spectral gap and that $\lim_{n \to \infty} \rho(n^2)^{1/n} = 1$ (which includes, in particular, the case of rapid decay).  Then $A$ has the Dixmier averaging property, and so in particular it is simple and has a unique trace.
	\end{introtheorem}
	
	In the case of unitary generators $(u_j)_{j=1}^d$, the proof of the theorem is much easier; one simply defines the averaging operator $Ty = \frac{1}{2d} \sum_{j=1}^d (u_j y u_j^* + u_j^* y u_j)$ and then applies the rapid decay estimate to $\norm{T^m y - \tau(y)}$ for $y$ a polynomial in the $u_j$'s (and semi-rapid decay would be sufficient as well).  This was already exploited for group $\mathrm{C}^*$-algebras in \cite[proof of Theorem 4.4.5]{Olesen2016}; here spectral gap for the group generators is equivalent to $G$ not being inner amenable (see e.g.\ \cite[\S 4]{Vaes2012}).
	
	For non-unitary generators, we can first assume without loss of generality that the $x_j$'s are self-adjoint with $\norm{x_j} \leq 1$ and then use the unitaries $u_j = \exp(ix_j)$ to define an associated averaging operator $T$.  The challenge is that we cannot immediately apply the rapid decay estimate to $T^m y$ since $u_j$ is not a polynomial in the given generators $x_j$. Hence, we must truncate the power series of the exponential $e^{ix_n}$ and estimate the errors carefully (see \S \ref{sec: Dixmier proof}).
	
	In order to apply Theorem \ref{thm: general Dixmier averaging} to the $q$-Gaussian family, we need both rapid decay and spectral gap.  The rapid decay property was established by \cite{Bozejko1999ultracontractivity}.  We obtain spectral gap from a general free probabilistic criterion (for definitions and precise constants, see \S \ref{sec: free probability} below).
	
	\begin{introtheorem}[{See Theorem \ref{thm: Lipschitz conjugate variables and spectral gap 2}}] \label{thm: Lipschitz conjugate variables and spectral gap}
		Let $(M,\tau)$ be a tracial von Neumann algebra generated by a self-adjoint tuple $x = (x_1,\dots,x_d)$ with $d \geq 2$.  If $x$ has Lipschitz conjugate variables, then $(M,\tau,x)$ has spectral gap.
	\end{introtheorem}
	
	Although this has not been explicitly noted in the literature, it is quite close to the results of Dabrowski.  Indeed, in the process of showing that the existence of conjugate variables implies the von Neumann algebra is non-Gamma, Dabrowski notes in \cite[Remark 11]{Dabrowski2014} that the spectral gap would follow from having bounded second-order conjugate variables.  In turn, the methods of \cite{Dabrowski2014} imply that bounded second-order conjugate variables exist provided that $x$ has Lipschitz conjugate variables, which Diez has also shown~\cite[Proposition 2.14]{Diez2026}.  In the interest of clarity, we give in \S \ref{sec: free probability} a self-contained argument for Theorem \ref{thm: Lipschitz conjugate variables and spectral gap} combining parts of \cite[proof of Lemma 9]{Dabrowski2010} with an identity of Mai \cite[Corollary V.5.6]{MaiThesis} (which was also generalized to the operator-valued setting by Lee \cite[Lemma 3.2]{Lee2024conjugate}).
	
	Finally, to apply Theorem \ref{thm: Lipschitz conjugate variables and spectral gap} to the $q$-Gaussians, we invoke the recent result of Miyagawa and Speicher \cite{MiyagawaSpeicher2023} that a finite $q$-Gaussian family has Lipschitz conjugate variables for all $q \in (-1,1)$ (this was shown earlier for small $q$ by Dabrowski \cite[\S 4.3]{Dabrowski2014}); see also \cite{KSW2023factoriality,Kumar2023conjugate,Yang2024twisted} for generalizations to the deformed Araki--Woods setting.  We thus conclude the following results about the $q$-Gaussian $\mathrm{C}^*$-algebras.
	
	\begin{introtheorem} \label{thm: q-Gaussian case}
		Let $q \in (-1,1)$.  Let $d \in \{2,3,\dots\} \cup \{\infty\}$.  Let $A$ be the $\mathrm{C}^*$-algebra generated by a $q$-Gaussian family $x = (x_j)_{j=1}^d$.  Then $A$ has the Dixmier averaging property, has a unique trace, and is simple.
	\end{introtheorem}
	
	Assuming Theorems \ref{thm: general Dixmier averaging} and \ref{thm: Lipschitz conjugate variables and spectral gap}, the proof for Theorem~\ref{thm: q-Gaussian case} is short, and therefore, we include it here.
	\begin{proof}[Proof of Theorem \ref{thm: q-Gaussian case}]
		As mentioned above, unique trace and $\mathrm{C}^*$-simplicity follow from Dixmier averaging.  It is straightforward to check that the Dixmier averaging property is preserved by inductive limits.  Hence, it suffices to prove the claim for $d < \infty$.
		
		The work of Bo{\.z}ejko \cite[p. 210, Main Theorem, item (a)]{Bozejko1999ultracontractivity} shows that $A = \mathrm{C}^*(x_1,\dots,x_d)$ has rapid decay.  Moreover, the work of Miyagawa and Speicher \cite[Theorem 1.1]{MiyagawaSpeicher2023} shows that $(x_j)_{j=1}^d$ has Lipschitz conjugate variables, which implies spectral gap by Theorem \ref{thm: Lipschitz conjugate variables and spectral gap}.  Hence, $A$ has Dixmier averaging by Theorem \ref{thm: general Dixmier averaging}.
	\end{proof}
	
	\begin{remark}
		Theorem \ref{thm: q-Gaussian case} also generalizes to mixed $q$-Gaussian algebras.  The existence of Lipschitz conjugate variables follows from \cite[Remark 5.4]{MiyagawaSpeicher2023} or \cite{Kumar2023conjugate}.  The rapid decay estimate of Bo{\.z}ejko also generalizes to the mixed $q$-Gaussian setting with a similar proof, as shown by Kr{\.o}lak \cite[Theorem 29]{Krolak2005}.
	\end{remark}
	
	The rest of the paper is organized into two sections:  \S \ref{sec: Dixmier proof}, which proves Theorem \ref{thm: general Dixmier averaging}, and \S \ref{sec: free probability}, which proves Theorem \ref{thm: Lipschitz conjugate variables and spectral gap}.
	
	\subsection*{Acknowledgements}
	
	We thank Akihiro Miyagawa for pointing us to the estimates in Mai's thesis used in \S \ref{sec: free probability}.  We also thank Charles-Philippe Diez, Srivatsav Kunnawalkam Elayavalli, and Brent Nelson for pointing out additional references.  We thank IMPAN for hosting DJ's visit in May 2026, where this project began; travel funding was provided by the National Science Center, Poland (NCN) grant no. 2021/43/D/ST1/01446.
	
	\subsection*{Funding}
	
	DJ was supported by an EU Horizon Marie Sk{\l}odowska Curie Action\footnote{Views and opinions expressed are those of the author(s) only and do not necessarily reflect those of the European Union or the Research Executive Agency. Neither the European Union nor the granting authority can be held responsible for them.}, FREEINFOGEOM, grant id: 101209517. MW was partially supported by the National Science Center, Poland (NCN) grant no. 2021/43/D/ST1/01446.
	
	\section{Proof of Theorem \ref{thm: general Dixmier averaging} on Dixmier averaging} \label{sec: Dixmier proof}
	
	To leverage the spectral gap property in the context of Dixmier averaging, it is simpler to work with unitary elements. The following preliminary lemma allows us to pass the spectral gap property freely between self-adjoint generators and their unitary exponentials.
	
	\begin{lemma}
		Let $A$ be a $\mathrm{C}^*$-algebra with a faithful trace $\tau$.  Let $(x_j)_{j=1}^d$ be self-adjoint elements of $A$ with $\norm{x_j} \leq 1$.  Let $u_j = \exp(ix_j)$ and $u = (u_j)_{j=1}^d$.  Then $(A,\tau,x)$ has spectral gap if and only if $(A,\tau,u)$ has spectral gap.
	\end{lemma}
	
	\begin{proof}
		We need to relate the $L^2$-derivations associated with $x$ and $u$. Since $u_j = \sum_{k=0}^\infty \frac{i^k}{k!} x_j^k$, the commutator with any element $y \in A$ is given by
		\[ [u_j, y] = \sum_{k=1}^\infty \frac{i^k}{k!} \sum_{l=0}^{k-1} x_j^l [x_j, y] x_j^{k-1-l}. \]
		Taking the $L^2$-norm and using the fact that $\norm{x_j} \leq 1$, we obtain that
		\begin{align*} \norm{[u_j, y]}_{\tau} &\leq \sum_{k=1}^\infty \frac{1}{k!} k \norm{x_j}^{k-1} \norm{[x_j, y]}_{\tau} \\&\leq \left( \sum_{k=1}^\infty \frac{1}{(k-1)!} \right) \norm{[x_j, y]}_{\tau} \\&= e \norm{[x_j, y]}_{\tau}. \end{align*}
		A similar bound holds for $\norm{[u_j^*, y]}_{\tau}$. Therefore, if $x$ has a spectral gap, $u$ also has a spectral gap.
		
		Conversely, we can recover $x_j$ via the principal logarithm so that $$x_j = -i \log(u_j) = -i \sum_{k=1}^\infty \frac{(-1)^{k-1}}{k} (u_j - 1)^k.$$   Note that for $t \in [-1,1]$, $|1 - e^{it}|^2 = 2 - 2 \cos t \leq 2 - 2 \cos 1 < 1$.  Letting $r = \sqrt{2 - 2 \cos 1}$, we have by functional calculus that $\norm{u_j - 1} \leq r < 1$, ensuring the series converges absolutely in the operator norm. By expanding the commutators of $(u_j - 1)^k$ similarly, we can bound $\norm{[x_j, y]}_{\tau}$ by a constant multiple of $\norm{[u_j, y]}_{\tau}$ plus the bound on the adjoint. Hence, the spectral gaps are equivalent.
	\end{proof}
	
	With this equivalence in hand, we proceed to prove our first main result. The core strategy is to approximate the averaging operator via a truncated Taylor series while controlling the resulting uniform and $L^2$ errors.
	
	\begin{proof}[Proof of {Theorem \ref{thm: general Dixmier averaging}}]
		First, let $\re(x_j) = (x_j + x_j^*)/2$ and $\im(x_j) = (x_j - x_j^*)/2i$. Note that
		\[
		\norm{[x_j,y]}_{\tau} + \norm{[x_j^*,y]}_{\tau} \leq 2 \left[ \norm{[\re(x_j),y]}_{\tau} + \norm{[\im(x_j),y]}_{\tau} \right].
		\]
		Thus, by replacing $(x_1,\dots,x_d)$ by $(\re(x_1),\im(x_1),\dots,\re(x_d),\im(x_d))$, we can assume that our spectral gap tuple consists of self-adjoint elements. Note that rapid decay is also preserved because $*$-polynomials up to degree $n$ in $(x_1,\dots,x_d)$ are the same as $*$-polynomials up to degree $n$ in $(\re(x_1),\im(x_1),\dots,\re(x_d),\im(x_d))$. Thus, for the rest of the proof, we assume, without loss of generality, that $x_1$, \dots, $x_d$ are self-adjoint and, by rescaling, we also assume that $\norm{x_j} \leq 1$.
		
		Let $u_j = \exp(ix_j)$ for $j=1,\dots,d$. By the preceding lemma, the tuple of unitaries $u = (u_1, \dots, u_d)$ has spectral gap. Let $T: A \to A$ be the averaging operator defined by
		\[
		T(y) = \frac{1}{2d} \sum_{j=1}^d (u_j y u_j^* + u_j^* y u_j).
		\]
		The spectral gap property implies the existence of a constant $\gamma \in (0,1)$ such that for all $y \in A$ with $\tau(y) = 0$ we have
		$\norm{T(y)}_{\tau} \leq (1-\gamma)\norm{y}_{\tau}$ (see e.g.\ \cite[Lemma 4.4]{FJP2026}). By iterating $T$, we obtain $\norm{T^n(y)}_{\tau} \leq (1-\gamma)^n \norm{y}_{\tau}$ for all $n \in \N$. We now approximate the unitaries $u_j$ by their Taylor polynomials of degree $m$
		\[
		u_{j,m} = \sum_{k=0}^m \frac{i^k}{k!} x_j^k.
		\]
		Since the derivative of order $m+1$ of $e^{it}$ is bounded by $1$ for $t \in \R$, we have the Taylor remainder estimate
		\[
		\sup_{t \in [-1,1]} \left| e^{it} - \sum_{k=0}^m \frac{(it)^k}{k!} \right| \leq \frac{1}{(m+1)!}.
		\]
		Since $\norm{x_j} \leq 1$, the properties of functional calculus imply that
		\[
		\norm{u_j - u_{j,m}} \leq \frac{1}{(m+1)!}.
		\]
		Define the corresponding truncated averaging operator $T_m: A \to A$ by
		\[
		T_m(y) = \frac{1}{2d} \sum_{j=1}^d (u_{j,m} y u_{j,m}^* + u_{j,m}^* y u_{j,m}).
		\]
		If $y$ is a $*$-polynomial in $x$ of degree $\deg(y)$, then $T_m(y)$ is a $*$-polynomial of degree at most $\deg(y) + 2m$. Consequently, $\deg(T_m^n(y)) \leq \deg(y) + 2mn$. By the definition of the growth function $\rho$, we have
		\[
		\norm{T_m^n(y)} \leq \rho(\deg(y) + 2mn) \norm{T_m^n(y)}_{\tau}.
		\]
		Next, we estimate the uniform distance between $T^n$ and $T_m^n$. Since $u_j$ is unitary ($\norm{u_j} = 1$), the triangle inequality yields $$\norm{u_{j,m}} \leq \norm{u_j} + \norm{u_{j,m} - u_j} \leq 1 + \frac{1}{(m+1)!}.$$ Applying $T_m$ therefore scales the uniform norm by a factor of at most $\left(1 + \frac{1}{(m+1)!}\right)^2$. Using a standard telescoping sum over $n$ iterations, the cumulative truncation error is bounded by
		\begin{align*}
			\norm{T^n(y) - T_m^n(y)} &\leq \sum_{k=0}^{n-1} \norm{T_m^k (T - T_m) T^{n-1-k}(y)} \\&\leq C \left(1 + \frac{1}{(m+1)!}\right)^{2n} \frac{n}{(m+1)!} \norm{y}
		\end{align*}
		for some absolute constant $C > 0$. Setting $m = n$ to balance the bounds, we see that
		\begin{align*}
			\norm{T^n(y)} &\leq \norm{T_n^n(y)} + \norm{T^n(y) - T_n^n(y)} \\
			&\leq \rho(\deg(y) + 2n^2) \norm{T_n^n(y)}_{\tau} + C \left(1 + \frac{1}{(n+1)!}\right)^{2n} \frac{n}{(n+1)!} \norm{y} \\
			&\leq \rho(\deg(y) + 2n^2) \left( \norm{T^n(y)}_{\tau} + \norm{T_n^n(y) - T^n(y)}_{\tau} \right) \\&+ C \left(1 + \frac{1}{(n+1)!}\right)^{2n} \frac{n}{(n+1)!} \norm{y}.
		\end{align*}
		Since the $L^2$-norm is dominated by the uniform norm, we substitute $\norm{T^n(y)}_{\tau} \leq (1-\gamma)^n \norm{y}_{\tau}$ to find
		\[
		\norm{T^n(y)} \leq \rho(\deg(y) + 2n^2) \left[ (1-\gamma)^n \norm{y}_{\tau} + \text{Error} \right] + \text{Error},
		\]
		where $\text{Error} = C \left(1 + \frac{1}{(n+1)!}\right)^{2n} \frac{n}{(n+1)!} \norm{y}$.
		
		By hypothesis, $\lim_{n \to \infty} \rho(n^2)^{1/n} = 1$, meaning $\rho$ grows subexponentially. The term $(1-\gamma)^n$ decays exponentially, and the factorial term in the denominator ensures that the Error term decays super-exponentially. Therefore, the right-hand side converges to $0$ as $n \to \infty$, demonstrating that $\lim_{n \to \infty} \norm{T^n(y)} = 0$ for any trace-zero polynomial $y$.
		
		Since polynomials are dense in $A$, for any $y \in A$ and $\varepsilon > 0$, we can choose a trace-zero polynomial $y_0$ such that $\norm{(y - \tau(y)1) - y_0} < \varepsilon/2$. Because $T$ is contractive, $\norm{T^n(y) - \tau(y)1} \leq \norm{T^n(y_0)} + \varepsilon/2$. Choosing $n$ sufficiently large yields $\norm{T^n(y) - \tau(y)1} < \varepsilon$. This proves that $A$ possesses the Dixmier averaging property.
	\end{proof}
	While the simple choice of $m=n$ suffices to yield the required subexponential decay, a more delicate reworking of the error terms allows for a slightly sharper growth condition. We remark this below.
	\begin{remark}
		In the proof of Theorem \ref{thm: general Dixmier averaging}, we set $m=n$ for simplicity, which requires the subexponential growth condition $\lim_{n \to \infty} \rho(n^2)^{1/n} = 1$, or equivalently $\log \rho(k) = o(\sqrt{k})$. However, this can be slightly optimized by choosing $m$ to carefully balance the exponential decay from the spectral gap with the factorial decay from the Taylor remainder. 
		
		Let $1-\gamma = e^{-c}$ for some $c > 0$. The proof requires both the spectral gap term and the error term to decay when multiplied by the polynomial growth factor, meaning we need $\rho(2mn)e^{-cn} \to 0$ and $\rho(2mn)\frac{n}{(m+1)!} \to 0$. 
		
		Taking the logarithm of the error term and applying Stirling's approximation, the decay is dominated by $-\log((m+1)!) \sim -m \log m$. To balance the two decay rates, we equate the exponents so that $m \log m \approx c n$, which implies $m \approx \frac{cn}{\log n}$. 
		
		Substituting this choice into the degree $k = 2mn$, we get $k \approx \frac{2c n^2}{\log n}$, which inverts to $n \approx \sqrt{\frac{k \log k}{4c}}$. The necessary condition $\log \rho(k) \ll cn$ then translates to the optimal growth condition
		\[
		\log \rho(k) = o(\sqrt{k \log k}),
		\]
		which is a slightly weaker requirement than $o(\sqrt{k})$. Equivalently, this can be restated as $\lim_{m \to \infty} \rho(m^2 \log m)^{\frac{1}{m \log m}} = 1$.    
	\end{remark}
	
	\section{Proof of Theorem \ref{thm: Lipschitz conjugate variables and spectral gap} on spectral gap} \label{sec: free probability}
	
	In this section, we prove Theorem \ref{thm: Lipschitz conjugate variables and spectral gap}.  This result is essentially a combination of \cite[Lemma 9, Remark 11]{Dabrowski2010} and \cite[Corollary V.5.6]{MaiThesis} or \cite[Lemma 3.2]{Lee2024conjugate}. Still, for the reader's convenience, we give a streamlined, self-contained argument for the spectral gap inequality after recalling the relevant definitions.
	
	\begin{definition}[{Free difference quotients \cite[\S 3]{VoiculescuFE5}}]
		Let $\C \ip{t_1,\dots,t_d}$ be the $*$-algebra of formal noncommutative polynomials in self-adjoint indeterminates $t_1$, \dots, $t_d$.  Note that we view $\C \ip{t_1,\dots,t_d} \otimes \C \ip{t_1,\dots,t_d}$ as a bimodule over $\C \ip{t_1,\dots,t_d} \otimes \C \ip{t_1,\dots,t_d}$ by left multiplication on the left tensorand and right multiplication on the right tensorand.  Voiculescu's \emph{free difference quotient} $\partial_j: \C \ip{t_1,\dots,t_d} \to \C \ip{t_1,\dots,t_d} \otimes \C \ip{t_1,\dots,t_d}$ is the unique linear map satisfying
		\begin{equation} \label{eq: derivative of x}
			\partial_j t_k = \delta_{j,k} 1 \otimes 1
		\end{equation}
		and the Leibniz rule (i.e., it is a derivation),
		\begin{equation} \label{eq: Leibniz rule}
			\partial_j(fg) = f \cdot \partial_j g + \partial_j f \cdot g.
		\end{equation}
	\end{definition}
	With the free difference quotients established, we can formalize a noncommutative integration-by-parts formula with respect to the trace. The vectors in $L^2(M,\tau)$ that implement this formula act as the adjoints to our derivations and are known as conjugate variables.
	\begin{definition}[{Conjugate variables \cite[Definition 3.1]{VoiculescuFE5}}]
		Let $(M,\tau)$ be a tracial von Neumann algebra generated by self-adjoints $x = (x_1,\dots,x_d)$.  For $f, g \in \C\ip{t_1,\dots,t_d}$, let $f(x)$ denote the evaluation on $x$ and $(f \otimes g)(x) = f(x) \otimes g(x)$.  We say that $\xi = (\xi_1,\dots,\xi_d) \in L^2(M,\tau)_{\sa}^d$ is family of \emph{conjugate variables} for $x$ if for all $f \in \C\ip{t_1,\dots,t_d}$,
		\begin{equation} \label{eq: conjugate variable relation}
			\ip{\xi_j, f(x)}_\tau = \tau \otimes \tau(\partial_j f(x)).
		\end{equation}
	\end{definition}
	
	If $(M,\tau,x)$ admits conjugate variables $\xi$, then $x = (x_1,\dots,x_d)$ are algebraically free \cite{MSW2017}, and therefore the evaluation map $\C\ip{t_1,\dots,t_d} \to M$, $f \mapsto f(x)$ is an isomorphism onto its image, which we denote by $\C\ip{x_1,\dots,x_d}$.  We can therefore view $\partial_j$ as a densely defined operator
	\[
	L^2(M,\tau) \supseteq \C\ip{x_1,\dots,x_d} \to L^2(M,\tau) \otimes L^2(M,\tau),
	\]
	and the definition of conjugate variables means that
	\[
	1 \otimes 1 \in \dom(\partial_j^*), \qquad \xi_j = \partial_j^*(1 \otimes 1).
	\]
	Algebraic computations from \eqref{eq: Leibniz rule} and the relation
	\[
	\partial_j(f^*) = \operatorname{flip}((\partial_j f)^*),
	\]
	where $\operatorname{flip}$ is the tensor flip, imply that
	\[
	\C\ip{x_1,\dots,x_d} \otimes \C\ip{x_1,\dots,x_d} \subseteq \dom(\partial_j^*),
	\]
	and
	\begin{equation} \label{eq: partial star formula}
		\partial_j^*(f(x) \otimes g(x)) = f(x) \xi_j g(x)  - (\id \otimes \tau)(\partial_j f(x)) g(x) - f(x)(\tau \otimes \id)(\partial_jg(x)).
	\end{equation}
	See \cite[Propositions 4.1, 4.3]{VoiculescuFE5}.  Moreover, since the domain of $\partial_j^*$ is dense, this means that $\partial_j$ is closable; we denote the closure by $\overline{\partial}_j$.  Furthermore, we note that if $M^{\op}$ denotes the opposite von Neumann algebra, then $(M \overline{\otimes} M^{\op},\tau \otimes \tau^{\op})$ is a tracial von Neumann algebra and
	\[
	L^2(M \overline{\otimes} M^{\op},\tau \otimes \tau^{\op}) \cong L^2(M,\tau) \otimes L^2(M,\tau).
	\]
	For our purposes, the mere existence of conjugate variables is not enough. In addition, we require analytic regularity. Specifically, when these variables fall within the domain of the closures $\overline{\partial}_j$ and yield bounded derivations, they provide the exact structural rigidity needed to force a spectral gap.
	\begin{definition}[Lipschitz conjugate variables {\cite[Definition 1]{Dabrowski2014}}]
		Let $(M,\tau)$ be a tracial von Neumann algebra generated by self-adjoints $x = (x_1,\dots,x_d)$.  Then $x$ is said to have \emph{Lipschitz conjugate variables} if it has conjugate variables $\xi$ such that $\xi_j$ is in the domain of $\overline{\partial}_k$ and $\overline{\partial}_k \xi_j$ is a element of $M \otimes M^{\operatorname{op}} \subseteq L^2(M \otimes M^{\operatorname{op}}) \cong L^2(M,\tau) \otimes L^2(M,\tau)$, for every $j, k \in \{1,\dots,d\}$.
	\end{definition}
	
	We next recall the following result from Dabrowski.  In the statement, we use the notation $[x,h]$ whenever $x \in M$ and $h$ is in a bimodule over $M$ such as $L^2(M,\tau) \otimes L^2(M,\tau)$.  We also use the bilinear map $\#: M \otimes M \times H \to H$ given by the linear extension of $(x \otimes y) \# h = xhy$.
	
	\begin{proposition}[{See \cite[Proof of Lemma 9]{Dabrowski2010}}] \label{prop: Dabrowski identity}
		Let $(M,\tau)$ be a tracial von Neumann algebra generated by $x = (x_1,\dots,x_d)$.  Then
		\[
		2(d-1) \norm{f - \tau(f)}_\tau^2 = \sum_{j=1}^d \ip{[f,1 \otimes 1], \partial_j ([f,x_j])}_{\tau \otimes \tau} - \sum_{j=1}^d \ip{[[f,x_j],1 \otimes 1],\partial_j f}_{\tau \otimes \tau}
		\]
	\end{proposition}
	
	\begin{proof}
		First, we note that $\norm{[f,1\otimes 1]}_{\tau \otimes \tau}^2 = \norm{f \otimes 1 - 1 \otimes f}_{\tau \otimes \tau}^2 = 2 \norm{f - \tau(f)}_\tau^2$, and so the left-hand side is $(d-1) \norm{[f,1\otimes 1]}_{\tau \otimes \tau}^2$.  From the Leibniz rule,
		\[
		\partial_j [f,x_j] = [\partial_j f, x_j] + [f,1 \otimes 1],
		\]
		so that
		\begin{equation} \label{eq: number 1}
			\norm{[f, 1 \otimes 1]}_{\tau \otimes \tau}^2 = \ip{[f,1\otimes 1], \partial_j [f,x_j]}_{\tau \otimes \tau} - \ip{[f,1 \otimes 1], [\partial_j f, x_j]}_{\tau \otimes \tau}.
		\end{equation}
		Focusing on the last term, we use the Jacobi identity to write
		\begin{align} 
			\ip{[f,1 \otimes 1], [\partial_j f, x_j]}_{\tau \otimes \tau} &= \tau \otimes \tau ([f^*,[\partial_j f, x_j]]) \label{eq: number 2} \\
			&= \tau \otimes \tau([[x_j,f^*],\partial_j f]) + \tau \otimes \tau([[f^*,\partial_j f],x_j]). \label{eq: number 3}
		\end{align}
		Focusing on the last term again, we use the traciality of $\tau$ to show that
		\begin{align}
			\tau \otimes \tau &([[f^*,\partial_j f],x_j]) \label{eq: number 4} \\
			&= \tau \otimes \tau\Bigl(f^* (\partial_j f) x_j - x_j f^* (\partial_j f) - (\partial_j f) f^* x_j + x_j (\partial_j f) f^* \Bigr) \label{eq: number 5} \\
			&= -\tau \otimes \tau( [f^*, \partial_j f \# [x_j, 1 \otimes 1]]). \label{eq: number 6}
		\end{align}
		More explicitly, for instance, the first term $\tau \otimes \tau(f^* (\partial_j f) x_j)$ in \eqref{eq: number 4} is equal to $\tau \otimes \tau(f^* \partial_j f \# (1 \otimes x_j))$, using the identity
		\[
		\tau \otimes \tau( f^*(a \otimes b) x_j) = \tau \otimes \tau(f^*a \otimes x_j b) = \tau \otimes \tau(f^* (a \otimes b) \# (1 \otimes x_j)),
		\]
		which is then extended linearly from the simple tensors $a \otimes b$, and the other terms in \eqref{eq: number 2} are handled similarly.  This allows us to apply the well-known identity \cite[eq. (8.3)]{MS2017}
		\begin{equation} \label{eq: number 7}
			\sum_{j=1}^d \partial_j f \# [x_j, 1 \otimes 1] = [f, 1 \otimes 1],
		\end{equation}
		which can be verified by checking it when $f(x) = x_k$ and noting that both sides are derivations in $f$.  We thus have from \eqref{eq: number 4} that
		\begin{align*}
			\sum_{j=1}^d \tau \otimes \tau([[f^*,\partial_j f],x_j]) &= -\sum_{j=1}^{d} \tau \otimes \tau([f^*, \partial_j f \# [x_j, 1 \otimes 1]]) \\
			&= - \tau \otimes \tau([f^*, [f, 1 \otimes 1]]) \\
			&= - \norm{[f,1\otimes 1]}_{\tau \otimes \tau}^2.
		\end{align*}
		Combining this with \eqref{eq: number 2},
		\begin{align*}
			\sum_{j=1}^d \ip{[f,1 \otimes 1], [\partial_j f, x_j]}_{\tau \otimes \tau} &= \sum_{j=1}^d \tau \otimes \tau\Bigl([[x_j,f^*],\partial_j f] + [[f^*,\partial_j f],x_j] \Bigr) \\
			&= \sum_{j=1}^d \ip{ [[f,x_j], 1 \otimes 1], \partial_j f}_{\tau \otimes \tau} - \norm{[f,1\otimes 1]}_{\tau \otimes \tau}^2.
		\end{align*}
		We then substitute this into \eqref{eq: number 1} and obtain
		\begin{multline*}
			d \norm{[f, 1 \otimes 1]}_{\tau \otimes \tau}^2 \\
			= \sum_{j=1}^d \ip{[f,1 \otimes 1], \partial_j ([f,x_j])}_{\tau \otimes \tau} - \sum_{j=1}^d \ip{[[f,x_j],1 \otimes 1],\partial_j f}_{\tau \otimes \tau} + \norm{[f,1 \otimes 1]}_{\tau \otimes \tau}^2,
		\end{multline*}
		and subtracting $\norm{[f,1 \otimes 1]}_{\tau \otimes \tau}^2$ yields the desired formula.
	\end{proof}
	Both terms on the right-hand side of Proposition~\ref{prop: Dabrowski identity} are pairings
	of the form $\ip{[g,1\otimes 1], \partial_j h}_{\tau \otimes \tau}$, with $g$ and $h$ taken from
	$\{f, [f,x_j]\}$. To bound such a pairing, it suffices to control the adjoint $\partial_j^*$ on
	elements of the form $g \otimes 1$, since $\ip{g \otimes 1, \partial_j h}_{\tau \otimes \tau}
	= \ip{\partial_j^*(g \otimes 1), h}_\tau$. The following identity of Mai supplies exactly this
	control once the conjugate variables are Lipschitz.
	\begin{proposition}[{See \cite[Corollary V.5.6]{MaiThesis} or \cite[Lemma 3.2]{Lee2024conjugate}}] \label{prop: Mai identity}
		Let $(M,\tau)$ be a tracial von Neumann algebra generated by self-adjoints $x = (x_1,\dots,x_d)$, and assume that $x$ has conjugate variables $\xi$.  Then for $f \in \C\ip{x_1,\dots,x_d}$, we have
		\begin{equation} \label{eq: norm of partial star}
			\norm{\partial_j^*(f \otimes 1)}_\tau^2 = \ip{\partial_j^*(f^*f \otimes 1), \partial_j^*(1 \otimes 1)}_\tau.
		\end{equation}
		In particular, if the conjugate variables are Lipschitz, then
		\begin{equation} \label{eq: norm of partial star 2}
			\norm{\partial_j^*(f \otimes 1)}_\tau \leq \norm{ \overline{\partial}_j \xi_j}_{M \overline{\otimes} M^{\op}}^{1/2} \norm{f}_\tau.
		\end{equation}
	\end{proposition}
	
	\begin{proof}
		We first note the following consequence of \eqref{eq: partial star formula}:
		\[
		\partial_j^* (f^* f \otimes 1) = f^* \partial_j^*(f \otimes 1) - (\id \otimes \tau)(\partial_j(f^*) f). 
		\]
		We therefore have
		\begin{equation} \label{eq: derivative formula 1}
			\ip{\partial_j^*(f^*f \otimes 1), \partial_j^*(1 \otimes 1)}_\tau = \ip{f^* \partial_j^*(f \otimes 1), \partial_j^*(1 \otimes 1)}_\tau - \ip{(\id \otimes \tau)(\partial_j(f^*) f), \partial_j^*(1 \otimes 1)}_\tau,
		\end{equation}
		and meanwhile
		\begin{equation} \label{eq: derivative formula 2}
			\ip{\partial_j^*(f \otimes 1), \partial_j^*(f \otimes 1)}_\tau = \ip{\partial_j^*(f \otimes 1), f \partial_j^*(1 \otimes 1)}_\tau - \ip{\partial_j^*(f \otimes 1), (\id \otimes \tau)(\partial_j f)}_\tau.
		\end{equation}
		We note that the first term on the right-hand side of \eqref{eq: derivative formula 1} is equal to the corresponding term in \eqref{eq: derivative formula 2}.  Therefore, to prove \eqref{eq: norm of partial star}, it suffices to show that
		\[
		\ip{(\id \otimes \tau)(\partial_j(f^*) f), \partial_j^*(1 \otimes 1)}_\tau = \ip{\partial_j^*(f \otimes 1), (\id \otimes \tau)(\partial_j f)}_\tau,
		\]
		which is equivalent to
		\[
		\ip{\partial_j (\id \otimes \tau)(\partial_j(f^*) f), 1 \otimes 1}_{\tau \otimes \tau} = \ip{f \otimes 1, \partial_j(\id \otimes \tau)(\partial_j f)}_{\tau \otimes \tau},
		\]
		and in turn
		\[
		\tau \otimes \tau(\partial_j(\id \otimes \tau)(\partial_j(f^*)f)^*) = \tau \otimes \tau(f^* \partial_j( \id \otimes \tau)(\partial_j f)).
		\]
		To prove this, one simply uses the identity $\partial_j(f^*) = \operatorname{flip}((\partial_j f)^*)$ and the coassociativity relation $(\id \otimes \partial_j) \circ \partial_j = (\partial_j \otimes \id) \circ \partial_j$ (see e.g.\ \cite[eq. (8.2)]{MS2017}).  Indeed,
		\begin{align*}
			\tau \otimes \tau((\partial_j(\id \otimes \tau)(\partial_j(f^*)f))^*) &= \tau \otimes \tau((\partial_j(\id \otimes \tau)(\partial_j(f^*)f)^*)) \\
			&= \tau \otimes \tau(\partial_j(\tau \otimes \id)(f^* \partial_j f)) \\
			&= \tau \otimes \tau \otimes \tau(f^* (\id \otimes \partial_j) \circ \partial_j f) \\
			&= \tau \otimes \tau \otimes \tau(f^* (\partial_j \otimes \id) \circ \partial_j f) \\
			&= \tau \otimes \tau(f^* \partial_j(\id \otimes \tau)(\partial_j f)).
		\end{align*}
		This completes the proof of \eqref{eq: norm of partial star}.
		
		As a consequence of \eqref{eq: norm of partial star}, we have
		\[
		\norm{\partial_j^*(f \otimes 1)}_\tau^2 = \ip{f^* f \otimes 1, \partial_j \partial_j^*(1 \otimes 1)}_{\tau \otimes \tau} = \ip{f^*f \otimes 1, \overline{\partial}_j \xi_j}_{\tau \otimes \tau}.
		\]
		Moreover,
		\begin{align*}
			\ip{f^*f \otimes 1, \overline{\partial}_j \xi_j}_{\tau \otimes \tau} &= \tau \otimes \tau^{\op}((f^*f \otimes 1)(\partial_j \xi_j)) \\
			&\leq \norm{\overline{\partial}_j \xi_j}_{M \overline{\otimes} M^{\op}} (\tau \otimes \tau^{\op})(f^*f \otimes 1) \\
			&= \norm{\overline{\partial}_j \xi_j}_{M \overline{\otimes} M^{\op}} \norm{f}_\tau^2,
		\end{align*}
		which proves \eqref{eq: norm of partial star 2}.  
	\end{proof}
	
	\begin{remark}
		In \eqref{eq: norm of partial star 2}, one can also replace $\norm{\partial_j \xi_j}_{M \overline{\otimes} M^{\op}}$ with $\norm{(\id \otimes \tau)(\partial_j \xi_j)}_M$ since
		\[
		\ip{f^*f \otimes 1, \overline{\partial}_j \xi_j}_{\tau \otimes \tau} = \ip{f^*f, (\id \otimes \tau)(\overline{\partial}_j \xi_j)}_\tau.
		\]
		Hence, Corollary \ref{cor: main derivative estimate} and Theorem \ref{thm: Lipschitz conjugate variables and spectral gap 2} also hold with $\norm{(\id \otimes \tau)(\overline{\partial}_j \xi_j)}_M$ in place of $\norm{\overline{\partial}_j \xi_j}_{M \overline{\otimes} M^{\op}}$.
	\end{remark}
	
	With Proposition \ref{prop: Mai identity} in hand, we now have the required tools to bound the cross-terms arising in Dabrowski's identity. The following corollary translates the $L^2$-norm estimates of the Lipschitz conjugate variables into bounds for these pairings.
	\begin{corollary} \label{cor: main derivative estimate}
		Let $(M,\tau)$ be a tracial von Neumann algebra generated by self-adjoints $x = (x_1,\dots,x_d)$, and assume that $x$ has Lipschitz conjugate variables $\xi$.  Then for $f, g \in \C\ip{x_1,\dots,x_d}$, we have
		\[
		|\ip{[f, 1 \otimes 1], \partial_j g}_{\tau \otimes \tau}| \leq 2 \norm{\overline{\partial}_j \xi_j}_{M \overline{\otimes} M^{\op}}^{1/2} \norm{f}_\tau \norm{g}_\tau.
		\]    
	\end{corollary}
	
	\begin{proof}
		Write $[f, 1 \otimes 1] = f \otimes 1 - 1 \otimes f$.  Then
		\[
		\ip{f \otimes 1, \partial_j g}_{\tau \otimes \tau} = \ip{\partial_j^*(f \otimes 1), g}_\tau,
		\]
		and then we apply the Cauchy--Schwarz inequality and the bound \eqref{eq: norm of partial star 2} for $\partial_j^*(f \otimes 1)$.  For the other inequality, we apply the tensor flip identity to write
		\[
		\ip{1 \otimes f, \partial_j g}_{\tau \otimes \tau} = \ip{\partial_j (g^*), f^* \otimes 1},
		\]
		and then apply the previous bound together with $\norm{f^*}_\tau = \norm{f}_\tau$ and $\norm{g^*}_\tau = \norm{g}_\tau$.
	\end{proof}
	
	We now conclude the proof of the following more precise version of Theorem \ref{thm: Lipschitz conjugate variables and spectral gap}.
	
	\begin{theorem} \label{thm: Lipschitz conjugate variables and spectral gap 2}
		Let $(M,\tau)$ be a tracial von Neumann algebra generated by self-adjoints $x = (x_1,\dots,x_d)$ with $d \geq 2$, and assume that $x$ has Lipschitz conjugate variables $\xi$.  Then for $z \in M$,
		\begin{align}
			\norm{z - \tau(z)}_\tau &\leq \frac{2}{d-1} \sum_{j=1}^d \norm{\overline{\partial}_j \xi_j}_{M \overline{\otimes} M^{\op}}^{1/2} \norm{[z,x_j]}_\tau \label{eq: free spectral gap estimate} \\
			&\leq \frac{2}{d-1} \left( \sum_{j=1}^d \norm{\overline{\partial}_j \xi_j}_{M \overline{\otimes} M^{\op}} \right)^{1/2} \left( \sum_{j=1}^d \norm{[z,x_j]}_\tau^2 \right)^{1/2} \label{eq: free spectral gap estimate 2}
		\end{align}
	\end{theorem}
	
	\begin{proof}
		It suffices to prove the claim for $z = f \in \C \ip{x_1,\dots,x_d}$ and we can assume without loss of generality that $\tau(f) = 0$.  From Proposition \ref{prop: Dabrowski identity}, we obtain
		\[
		2(d-1) \norm{f}_\tau^2 \leq \sum_{j=1}^d |\ip{[f,1 \otimes 1], \partial_j ([f,x_j])}_{\tau \otimes \tau}| + \sum_{j=1}^d |\ip{[[f,x_j],1 \otimes 1],\partial_j f}_{\tau \otimes \tau}|.
		\]
		Bounding each of the terms with Corollary \ref{cor: main derivative estimate}, we see that
		\[
		2(d-1) \norm{f}_\tau^2 \leq 4 \sum_{j=1}^d \norm{\overline{\partial}_j \xi_j}_{M \overline{\otimes} M^{\op}}^{1/2} \norm{[f,x_j]}_\tau \norm{f}_\tau,
		\]
		which yields \eqref{eq: free spectral gap estimate}, and then \eqref{eq: free spectral gap estimate 2} follows by the Cauchy--Schwarz inequality.
	\end{proof}

	\bibliography{name}

@article{archbold2017dixmier,
	title={The {D}ixmier property and tracial states for {$\mathrm{C}^*$}-algebras},
	author={Archbold, Robert and Robert, Leonel and Tikuisis, Aaron},
	journal={Journal of Functional Analysis},
	volume={273},
	number={8},
	pages={2655--2718},
	year={2017},
	publisher={Elsevier}
}

@article{Avsec2011,
	author        = {Avsec, Stephen},
	title         = {Strong Solidity of the $q$-{G}aussian Algebras for all $-1 < q < 1$},
	year          = {2011},
	journal        = {arXiv:1110.4918},
}

@article{BorstCaspersKlisseWasilewski2023,
	author = {Borst, Matthijs and Caspers, Martijn and Klisse, Mario and Wasilewski, Mateusz},
	title = {On the isomorphism class of {{\(q\)}}-{Gaussian} {{\(\mathrm{C}^\ast\)}}-algebras for infinite variables},
	fjournal = {Proceedings of the American Mathematical Society},
	journal = {Proc. Am. Math. Soc.},
	issn = {0002-9939},
	volume = {151},
	number = {2},
	pages = {737--744},
	year = {2023},
	language = {English},
	doi = {10.1090/proc/16165},
	zbMATH = {7630923},
	Zbl = {1511.46038}
}

@article{Bozejko1999ultracontractivity,
	author = {Marek Bo{\.z}ejko},
	title = {Ultracontractivity and strong Sobolev inequality for {$q$-Ornstein--Uhlenbeck} semigroup (-1 < q < 1)},
	journal = {Infinite Dimensional Analysis, Quantum Probability and Related Topics},
	volume = {2},
	number = {2},
	pages = {203-220},
	year = {1999},
	doi = {10.1142/S0219025799000114}
}

@article{BozejkoSpeicher1994,
	author = {Bo{\.z}ejko, Marek and Speicher, Roland},
	title = {An example of a generalized {Brownian} motion},
	fjournal = {Communications in Mathematical Physics},
	journal = {Commun. Math. Phys.},
	issn = {0010-3616},
	volume = {137},
	number = {3},
	pages = {519--531},
	year = {1991},
	language = {English},
	doi = {10.1007/BF02100275},
	zbMATH = {4190820},
	Zbl = {0722.60033}
}

@article{BozejkoKummererSpeicher1997,
	author = {Bo{\.z}ejko, Marek and K{\"u}mmerer, Burkhard and Speicher, Roland},
	title = {{{\(q\)}}-Gaussian processes: {Non}-commutative and classical aspects},
	fjournal = {Communications in Mathematical Physics},
	journal = {Commun. Math. Phys.},
	issn = {0010-3616},
	volume = {185},
	number = {1},
	pages = {129--154},
	year = {1997},
	language = {English},
	doi = {10.1007/s002200050084},
	zbMATH = {1019451},
	Zbl = {0873.60087}
}

@article{Caspers2023,
	author = {Caspers, Martijn},
	title = {On the isomorphism class of {{\(q\)}}-{Gaussian} {{\(W^\ast\)}}-algebras for infinite variables},
	fjournal = {Comptes Rendus. Math{\'e}matique. Acad{\'e}mie des Sciences, Paris},
	journal = {C. R., Math., Acad. Sci. Paris},
	issn = {1631-073X},
	volume = {361},
	pages = {1711--1716},
	year = {2023},
	language = {English},
	doi = {10.5802/crmath.489},
	zbMATH = {7811833},
	Zbl = {1545.46044}
}

@article{Dabrowski2010,
	author = {Yoann Dabrowski},
	title = {A note about proving non-$\Gamma$ under a finite
		non-microstates free Fisher information assumption},
	journal = {Journal of Functional Analysis},
	volume = {258},
	number = {11},
	pages = {3662-3674},
	year = {2010},
	doi = {10.1016/j.jfa.2010.02.010}
}

@article{Dabrowski2014,
	author = {Yoann Dabrowski},
	title = {A free stochastic partial differential equation},
	journal = {Ann. Inst. H. Poincar{\'e} Probab. Statist.},
	volume = {50},
	number = {4},
	pages = {1404-1455},
	year = {2014},
	doi = {10.1214/13-AIHP548}
}

@unpublished{Diez2026,
	author = {Charles-Philippe Diez},
	title = {Obata's rigidity theorem in free probability},
	note = {Preprint, arXiv:2603.05466},
	year = {2026}
}

@article{FJP2026,
	author = {Ilijas Farah and David Jekel and Jennifer Pi},
	title = {Quantum expanders and quantifier reduction for tracial von Neumann algebras},
	journal = {Journal of Symbolic Logic},
	volume = {91},
	number = {2},
	pages = {776-812},
	year = {2026},
	doi = {10.1017/jsl.2025.10100}
}

@article{FrischBourret1970,
	author = {Frisch, U. and Bourret, R.},
	title = {Parastochastics},
	fjournal = {Journal of Mathematical Physics},
	journal = {J. Math. Phys.},
	issn = {0022-2488},
	volume = {11},
	pages = {364--390},
	year = {1970},
	language = {English},
	doi = {10.1063/1.1665149},
	zbMATH = {3298792},
	Zbl = {0187.25902}
}

@article{GuionnetShlyakhtenko2014,
	author = {Guionnet, A. and Shlyakhtenko, D.},
	title = {Free monotone transport},
	fjournal = {Inventiones Mathematicae},
	journal = {Invent. Math.},
	issn = {0020-9910},
	volume = {197},
	number = {3},
	pages = {613--661},
	year = {2014},
	language = {English},
	doi = {10.1007/s00222-013-0493-9},
	zbMATH = {6358028},
	Zbl = {1312.46059}
}

@article{hayes2025selfless,
	title={Selfless reduced free product {$\mathrm{C}^*$}-algebras},
	author={Hayes, Ben and {Kunnawalkam Elayavalli}, Srivatsav and Robert, Leonel},
	journal={arXiv preprint arXiv:2505.13265},
	year={2025}
}

@article{Krolak2005,
	author = {Ilona Kr{\.o}lak},
	title = {Contractivity properties of Ornstein-Uhlenbeck semigroup for general commutation relations},
	journal = {Math. Z.},
	volume = {250},
	pages = {915-937},
	year = {2005},
	doi = {10.1007/s00209-005-0801-1}
}

@article{Kumar2023conjugate,
	author = {Kumar, Manish},
	title = {Conjugate variables approach to mixed $q$-Araki–Woods algebras: Factoriality and non-injectivity},
	journal = {Journal of Mathematical Physics},
	volume = {64},
	number = {9},
	pages = {093506},
	year = {2023},
	doi = {10.1063/5.0158660}
}

@article{KSW2023factoriality,
	author = {Manish Kumar and Adam Skalski and Mateusz Wasilewski},
	title = {Full Solution of the Factoriality Question for {$q$-Araki-Woods} von {N}eumann Algebras Via Conjugate Variables},
	journal = {Commun. Math. Phys.},
	volume = {402},
	pages = {157-167},
	year = {2023},
	doi = {10.1007/s00220-023-04734-5}
}

@unpublished{Lee2024conjugate,
	author = {Yoonkyeong Lee},
	title = {On conjugate systems with respect to completely positive maps},
	year = {2024},
	note = {Preprint, arXiv:2412.14336}
}

@phdthesis{MaiThesis,
	author = {Tobias Mai},
	title = {On the analytic theory of non-commutative distributions in free probability},
	school = {Universit{\"a}t des Saarlandes},
	year = {2017}
}

@article{MSW2017,
	title = {Absence of algebraic relations and of zero divisors under the assumption of full non-microstates free entropy dimension},
	journal = {Advances in Mathematics},
	volume = {304},
	pages = {1080-1107},
	year = {2017},
	doi = {10.1016/j.aim.2016.09.018},
	url = {https://www.sciencedirect.com/science/article/pii/S0001870815301183},
	author = {Tobias Mai and Roland Speicher and Moritz Weber}
}

@article{mandal2025simplicity,
	title={On the simplicity of certain mixed $q$-deformed Araki-Woods {C}{*}-algebras},
	author={Mandal, Malay and Mukherjee, Kunal and Patri, Issan},
	journal={Journal of Mathematical Analysis and Applications},
	volume={544},
	number={2},
	pages={129090},
	year={2025},
	publisher={Elsevier}
}

@book{MS2017,
	author = {James A. Mingo and Roland Speicher},
	title = {Free probability and random matrices},
	series = {Fields Institute Monographs},
	volume = {35},
	year = {2017},
	publisher = {Springer-Verlag},
	address = {New York}
}

@article{Miyagawa2023,
	author = {Miyagawa, Akihiro},
	title = {A short note on strong convergence of $q$-Gaussians},
	journal = {International Journal of Mathematics},
	volume = {34},
	number = {14},
	pages = {2350087},
	year = {2023},
	doi = {10.1142/S0129167X23500878}
}

@article{MiyagawaSpeicher2023,
	title = {A dual and conjugate system for q-Gaussians for all q},
	journal = {Advances in Mathematics},
	volume = {413},
	pages = {108834},
	year = {2023},
	issn = {0001-8708},
	doi = {10.1016/j.aim.2022.108834},
	author = {Akihiro Miyagawa and Roland Speicher}
}

@article{Nelson2015transport,
	author = {Brent Nelson},
	title = {Free Monotone Transport Without a Trace},
	journal = {Commun. Math. Phys.},
	volume = {334},
	pages = {1245-1298},
	year = {2015},
	doi = {10.1007/s00220-014-2148-0}
}

@article{Nou2004,
	author = {Nou, Alexandre},
	title = {Non-injectivity of the {{\(q\)}}-deformed von {Neumann} algebra},
	fjournal = {Mathematische Annalen},
	journal = {Math. Ann.},
	issn = {0025-5831},
	volume = {330},
	number = {1},
	pages = {17--38},
	year = {2004},
	language = {English},
	doi = {10.1007/s00208-004-0523-4},
	zbMATH = {2104994},
	Zbl = {1060.46051}
}

@phdthesis{Olesen2016,
	author = {Kristian Knudsen Olesen},
	title = {Analytic aspects of the Thompson groups},
	school = {University of Copenhagen},
	year = {2016},
	isbn = {978-87-7078-948-6}
}

@article{PSW2018,
	author = {Sandeepan Parekh and Koichi Shimada and Chenxu Wen},
	title = {Maximal amenability of the generator subalgebra in $q$-{G}aussian von {N}eumann algebras},
	journal = {J. Operator Theory},
	volume = {80},
	number = {1},
	year = {2018},
	pages = {125-152},
	doi = {10.7900/jot.2017jun28.2167}
}

@article{Ricard2005,
	author = {Ricard, {\'E}ric},
	title = {Factoriality of {{\(q\)}}-{Gaussian} von {Neumann} algebras},
	fjournal = {Communications in Mathematical Physics},
	journal = {Commun. Math. Phys.},
	issn = {0010-3616},
	volume = {257},
	number = {3},
	pages = {659--665},
	year = {2005},
	language = {English},
	doi = {10.1007/s00220-004-1266-5},
	zbMATH = {2221323},
	Zbl = {1079.81038}
}

@article{Vaes2012,
	title = {An inner amenable group whose von Neumann algebra does not have property {G}amma},
	author = {Stefaan Vaes},
	journal = {Acta Math.},
	volume = {208},
	number = {2},
	pages = {389-394},
	year = {2012},
	doi = {10.1007/s11511-012-0079-1}
}

@article{VoiculescuFE5,
	author = {Dan-Virgil Voiculescu},
	title = {The analogues of entropy and of {F}isher's information in free probability {V}},
	journal = {Inventiones Mathematicae},
	volume = {132},
	pages = {189-227},
	year = {1998},
	publisher = {Springer-Verlag},
	doi = {10.1007/s002220050222}
}

@article{Yang2024twisted,
	author = {Yang, Zhiyuan},
	title = {A Conjugate System for Twisted Araki–Woods von Neumann Algebras of Finite Dimensional Spaces},
	journal = {International Mathematics Research Notices},
	volume = {2024},
	number = {17},
	pages = {12044-12074},
	year = {2024},
	doi = {10.1093/imrn/rnae152}
}
	\bibliographystyle{alpha}
	
\end{document}